\documentclass[12pt]{amsart}
\usepackage{cite}
\usepackage{amsmath}
\usepackage{textgreek}
\usepackage[T1]{fontenc}
\usepackage[utf8]{inputenc}
\usepackage{lmodern}
\usepackage{microtype}
\usepackage{mathtools,amssymb,amsthm}
\usepackage{enumitem}
\usepackage{xcolor}
\usepackage[colorlinks=true,
  linkcolor=blue!55!black,
  citecolor=green!40!black,
  urlcolor=blue!60!black]{hyperref}
\usepackage[nameinlink,noabbrev]{cleveref}

\setlist{itemsep=2pt,topsep=5pt}
\newtheorem{theorem}{Theorem}[section]
\newtheorem{proposition}[theorem]{Proposition}
\newtheorem{lemma}[theorem]{Lemma}
\newtheorem{corollary}[theorem]{Corollary}
\theoremstyle{definition}
\newtheorem{definition}[theorem]{Definition}

\newcommand{\func}[1]{\operatorname{#1}}

\begin{document}
\title[The Exact Completion of the Category of Polish Groups]{The Exact
Completion of the Category of Polish Groups}
\author{Martino Lupini}
\address{Dipartimento di Matematica, Universit\`{a} di Bologna, Piazza di
Porta S. Donato, 5, 40126 Bologna,\ Italy}
\email{martino.lupini@unibo.it}
\urladdr{https://www.lupini.org/}
\thanks{The author was partially supported by the Starting Grant 101077154
\textquotedblleft Definable Algebraic Topology\textquotedblright\ from the
European Research Council, the Gruppo Nazionale per le Strutture Algebriche,
Geometriche e le loro Applicazioni (GNSAGA) of the Istituto Nazionale di
Alta Matematica (INDAM), and the University of Bologna. Part of this work
was done during a visit by the author to Chalmers University of Technology
and the University of Gothenburg. The hospitality of these institutions is
gratefully acknowledged.}
\subjclass[2020]{Primary 18E08, 54H11; Secondary 18A35, 18B10, 03E15}
\keywords{Regular category, exact category, quasi-abelian category, abelian
category, Polish group, group with a Polish cover, Borel-definable
homomorphism, approximately multiplicative lift, Baire category}
\date{\today }

\begin{abstract}
This article shows that the exact completion as a regular category of the
homological category of Polish groups and continuous group homomorphisms is
the category of groups with a Polish cover and Borel-definable group
homomorphisms. We also obtain several equivalent characterizations for the
morphisms in the exact completion. An analogous description is deduced for
several important subcategories of the category of Polish groups.
\end{abstract}

\setcounter{tocdepth}{1}

\maketitle
\tableofcontents

%\shorttitle{}

\section{Introduction}

In this paper we study the category of Polish groups from the viewpoint of
category theory. This work builds on \cite{lupini_looking_2024}, which
considered the category of Polish \emph{abelian} groups. That category was
shown to be quasi-abelian in the sense of \cite%
{schneiders_quasi-abelian_1999}. Its canonical completion to an abelian
category (left heart) was identified with the category of abelian groups
with a Polish cover and Borel-definable group homomorphisms, as defined in 
\cite{bergfalk_definable_2024,bergfalk_definable_2024-1}; see also \cite%
{casarosa_homological_2026,lupini_applications_2025}.

Abelian categories form the main framework for the development of abstract
homological algebra in the setting of category theory. Quasi-abelian
categories are a natural generalization, encompassing categories of
commutative algebraic structures (such as abelian groups or modules) endowed
with a topology. In such categories, epimorphisms can be distinct from \emph{%
regular} epimorphisms (corresponding to \emph{quotients}). The left heart
construction associates with a quasi-abelian category an abelian category in
which the missing quotients are added canonically.

This construction admits a natural generalization to the nonabelian context
of regular and Barr-exact categories. A regular category is, roughly
speaking, a finitely complete category with pullback-stable quotients in
which every arrow factors as a quotient followed by a monomorphism. The
kernel pair of a morphism is the categorical analogue of a compatible
equivalence relation, usually called a \emph{congruence} in universal
algebra.

A regular category is \emph{exact} if every congruence is effective, namely
if it is the kernel pair of its quotient. (Exact categories in this sense
are sometimes called more precisely Barr-exact, to distinguish them from
exact categories in the sense of Quillen, which provide a different
generalization of abelian categories.)

The exact completion of a regular category (also called ex/reg completion,
to distinguish it from the exact completion as a category with finite
limits) freely adds effective quotients of internal equivalence relations,
producing a Barr-exact category. When the input is quasi-abelian, this
construction recovers the left heart; see \Cref{sec:abelian-case} below.

In this paper we provide an explicit description of the exact completion of
the regular category of Polish groups, which is in fact a \emph{homological}
category, as the category of groups with a Polish cover and Borel-definable
homomorphisms. In particular, this shows that the category of groups with a
Polish cover is also homological, besides being exact. Thus, it falls just
short of being \emph{semi-abelian}, which would additionally require it to
have binary coproducts.

This recognition of the exact completion of the category of Polish groups
can be seen as the natural extension of the main result of \cite%
{lupini_looking_2024}, which proves the corresponding fact for \emph{abelian}
Polish groups. The main technical tool in \cite{lupini_looking_2024} was the
complexity-theoretic characterization of Polish subgroups of abelian Polish
groups from \cite{solecki_coset_2009}, in terms of the corresponding coset
equivalence relation. This characterization is not (currently) available for
arbitrary Polish groups, as it is an open problem whether the main results
of \cite{solecki_coset_2009} hold in general. Thus, we replace the appeal to 
\cite{solecki_coset_2009} with Brown's Polishability theorem for extensions
of Polish groups admitting a Borel section. Brown proved this theorem while
extending Moore's theory of measurable cocycles from locally compact groups
to arbitrary Polish groups \cite{brown_extensions_1971,moore_extensions_1964}%
.

More generally, our techniques allow us to prove analogous descriptions for
the exact completions of a host of \emph{thick} subcategories of the
category of Polish groups, such as the category of non-Archimedean Polish
groups, or TSI Polish groups.

We also establish a version for arbitrary groups with a Polish cover of the
existence theorem for approximately additive lifts from Section~5 of \cite%
{lupini_looking_2024}. Combining this result with classical results about
cross-sections for locally compact Polish groups, we generalize \cite[%
Corollaries~6.19 and~6.20]{lupini_looking_2024} to groups that are not
necessarily abelian.

The paper is divided into four sections after this introduction. In Section~%
\ref{Section:Polish} we recall fundamental facts pertaining to the theory of
Polish groups and groups with a Polish cover. In Section~\ref%
{Section:regular} we recall the definition of regular and exact categories.
In Section~\ref{Section:completion} we present the description of the exact
completion of the category of Polish groups as a regular category in terms
of groups with a Polish cover. Finally, in Section~\ref{Section:better-lifts}
we prove the existence theorem for approximately multiplicative lifts and
its locally compact and Lie consequences.

\subsubsection*{Acknowledgments}

The author is grateful to Ronnie Chen and Ivan Di Liberti for many useful
conversations and to the Association of Neurodivergent Logicians
(NeuroLogic) for its support. ChatGPT by OpenAI (version GPT-5.6 Sol) and
Claude by Anthropic (version Opus 5) were used in the preparation of this
manuscript.

\section{Polish groups\label{Section:Polish}}

In this section we recall fundamental notions pertaining to Polish groups,
as can be found in \cite{gao_invariant_2009,kechris_classical_1995}, and
groups with a Polish cover, as introduced and studied in \cite%
{lupini_looking_2024,bergfalk_definable_2024,bergfalk_definable_2024-1,casarosa_homological_2026}%
.

\subsection{Polish groups and groups with a Polish cover}

A \emph{Polish space} is a separable topological space whose topology is
induced by a complete metric. A \emph{Polish group} is a topological group
whose underlying space is Polish. Let $G$ be a Polish group. A subgroup $%
N\leq G$ is \emph{Polish} if it carries a Polish group topology for which
the inclusion $N\rightarrow G$ is continuous. Such a topology is unique, its
Borel sets are precisely those inherited from $G$, and $N$ is a Borel
subgroup of $G$; see \cite[Theorem~9.10]{kechris_classical_1995} and \cite[%
Section~2.2]{gao_invariant_2009}. In this case, the quotient set $G/N$ is a
homogeneous space with a Polish cover, considered as a pointed space with
basepoint corresponding to the trivial $N$-coset. When $N$ is also dense in $
%
% TODO [minor]: "phantom homogeneous Polish space" is defined here and never
% used again. Likewise \textbf{FDLCPAb} is introduced in Sec 2.2 but appears
% in neither Prop 4.2 nor Prop 4.3. Either use them or cut them.
G$, $G/N$ is a phantom homogeneous Polish space. When $N$ is furthermore 
\emph{normal} in $G$, the quotient $G/N$ is a \emph{group with a Polish cover%
}.

If $G/N$ is a homogeneous space with a Polish cover, then we define $%
E_{N}^{G}\subseteq G\times G$ to be the coset relation of $N$ in $G$. A
subspace with a Polish cover of $G/N$ is a pointed subspace of the form $H/N$%
, where $H$ is a Polish subgroup of $G$ containing $N$ (which implies that $%
N $ is a Polish subgroup of $H$).

\subsection{Classes of Polish groups\label{Subsection:classes}}

We now introduce several full replete subcategories of the category of
Polish groups. Membership in these subcategories is characterized as follows:

\begin{itemize}
\item $G\in\mathbf{nAP}$ if and only if $G$ is \emph{non-Archimedean}, i.e.,
it has a basis of identity neighborhoods consisting of subgroups \cite[%
Section~2.4]{gao_invariant_2009} or, equivalently, it has a compatible
ultrametric \cite[Theorem~2.4.1]{gao_invariant_2009};

\item $G\in\mathbf{TSI}$ if and only if $G$ is TSI, i.e., it has a
compatible two-sided invariant metric or, equivalently, a basis of
conjugation-invariant identity neighborhoods \cite{allison_dynamical_2021};

\item $G\in\mathbf{CLI}$ if and only if $G$ is CLI, i.e., it has a
compatible complete left-invariant metric \cite%
{lupini_games_2018,allison_class_2024};

\item $G\in\mathbf{proLieP}$ if and only if $G$ is pro-Lie, i.e., the limit
of a tower of (finite-dimensional)\ Lie groups with surjective continuous
homomorphisms as bonding maps;

\item $G\in\mathbf{nATSI}$ if and only if $G$ is both non-Archimedean and
TSI or, equivalently, it has a compatible two-sided invariant ultrametric
or, equivalently, a basis of identity neighborhoods that are normal
subgroups \cite[Theorem~1.1]{gao_non-archimedean_2014}; see also \cite[%
Theorem~2.6]{gao_graev_2013} and \cite[Proposition~2.2]%
{ding_non-archimedean_2017};

\item $G\in\mathbf{LCP}$ if and only if $G$ is locally compact;

\item $G\in\mathbf{Lie}$ if and only if $G$ is a finite-dimensional real Lie
group;

\item $G\in\mathbf{PAb}$ if and only if $G$ is abelian;

\item $G\in\mathbf{proLiePAb}$ if and only if $G$ is abelian and pro-Lie;

\item $G\in\mathbf{proLieTSI}$ if and only if $G$ is TSI and pro-Lie;

\item $G\in\mathbf{CP}$ if and only if $G$ is compact \cite%
{hofmann_structure_2013};

\item $G\in\mathbf{AmenLCP}$ if and only if $G$ is locally compact and
amenable \cite{paterson_amenability_1988};

\item $G\in\mathbf{SolP}$ if and only if $G$ is solvable as an abstract
group \cite{robinson_course_1982};

\item $G\in\mathbf{EnALCP}$ if and only if $G$ is locally compact,
non-Archimedean, and \emph{elementary} \cite{wesolek_elementary_2015};

\item for an extension-closed variety $\mathcal{F}$ of finite groups as in 
\cite[Section~2.1]{ribes_profinite_2010}, $G\in\mathbf{pro\text{-}\mathcal{F}%
P}$ if and only if $G$ is a compact pro-$\mathcal{F}$ group;

\item $G\in\mathbf{FDLCP}$ if and only if $G$ is locally compact and
finite-dimensional;

\item $G\in\mathbf{FDLCPAb}$ if and only if $G$ is abelian and belongs to $%
\mathbf{FDLCP}$;

\item $G\in\mathbf{LCPAb}_{\mathrm{cg}}$ if and only if $G$ is locally
compact, compactly generated, and abelian;

\item $G\in\mathbf{TorLCPAb}$ if and only if $G$ is an abelian locally
compact topological torsion group;

\item $G\in\mathbf{LCPAb}(p)$ if and only if $G$ is an abelian locally
compact topological $p$-group;

\item $G\in\mathbf{FLCPAb}$ if and only if $G$ is an abelian locally compact
group of finite ranks in the sense of \cite[Definitions~2.5 and~2.6]%
{hoffmann_homological_2007};

\item $G\in \mathbf{Hilb}$ if and only if $G$ is the additive group of a
real Hilbert space;

\item $G\in \mathbf{Ban}$ if and only if $G$ is the attive group of a real
Banach space.
\end{itemize}

Some of the above classes of abelian Polish groups have also been considered
in \cite{lupini_applications_2025,lupini_looking_2024}.

Thus $\mathbf{Ban}$ is the full replete subcategory of $\mathbf{Ban}$
spanned by the additive Polish groups underlying separable real Banach
spaces. A continuous homomorphism between additive real topological vector
spaces is automatically $\mathbb{R}$-linear, so the morphisms in $\mathbf{Ban%
}$ are precisely the bounded linear operators.

\subsection{Polish coset relations}

Let $G$ be a Polish group, and $N$ be a Polish normal subgroup of $G$. We
consider the conjugation action $G\curvearrowright N$.

The following lemma is \cite[Lemma 4.20]{lupini_looking_2024}. Note that the
group $A$ is assumed therein to be abelian, but this hypothesis is not used
in the proof. A similar argument is used in the proof of \cite[%
Proposition~3.5]{reid_chief_2022}.

\begin{lemma}
\label{Lemma:G-module} Let $G$ and $A$ be Polish groups. Suppose that $%
G\curvearrowright A$ is an action of $G$ on $A$ by automorphisms of $A$ that
is Borel separately in each variable. Then the action is continuous.
\end{lemma}

\begin{corollary}
\label{Corollary:continuous-conjugation} Let $G/N$ be a group with a Polish
cover. Then:

\begin{enumerate}
\item the conjugation action $G\curvearrowright N$ is continuous;

\item $E_{N}^{G}$ is a Polish subgroup of $G\times G$.
\end{enumerate}
\end{corollary}

\begin{proof}
(1) The conjugation action is Borel separately in each variable, since the
Borel structure on $N$ is inherited from $G$. Thus, (1) follows from Lemma~%
\ref{Lemma:G-module}.

(2) By (1) the product topology renders the semidirect product $N\rtimes G$
a Polish group. The homomorphism 
\begin{equation*}
N\rtimes G\rightarrow G\times G\text{, }\left( x,g\right) \mapsto \left(
g,xg\right)
\end{equation*}%
has $E_{N}^{G}$ as image.
\end{proof}

If $\mathcal{B}$ is a class of Polish groups that is closed under taking
closed normal subgroups and quotients by closed normal subgroups, we define
a group with a cover in $\mathcal{B}$ to be a group with a Polish cover $G/N$
where $G$, the coset relation $E_{N}^{G}$, and (hence) $N$ are in $\mathcal{B%
}$. For classes such as the TSI or the non-Archimedean TSI Polish groups,
the condition on $E_{N}^{G}$ records an additional regularity property of
the conjugation action; it is not merely a condition on $G$ and $N$
separately.

\section{Regular and exact categories\label{Section:regular}}

In this section we recall some notions concerning regular and exact
categories as can be found in \cite%
{barr_exact_1971,gran_introduction_2021,carboni_regular_1998,shulman_exact_2012}%
. Fundamental notions pertaining to category theory can be found in the
monographs \cite%
{awodey_category_2006,mac_lane_categories_1998,leinster_basic_2014}.

\subsection{Regular epimorphisms}

We recall the notion of \emph{regular epimorphism} in a category. Let $%
\mathcal{C}$ be a \emph{finitely complete} category. Given an arrow $%
f:D\rightarrow C$, one can consider the pullback diagram 
\begin{equation*}
\begin{array}{ccc}
D\times _{C}D & \overset{p_{2}}{\rightarrow } & D \\ 
p_{1}\downarrow &  & \downarrow f \\ 
D & \underset{f}{\rightarrow } & C%
\end{array}%
\end{equation*}%
This is called the \emph{kernel pair} of $f$. More generally, given arrows $%
f_{i}:D_{i}\rightarrow C$ for $i\in \left\{ 1,2\right\} $ one can consider
the \emph{pullback}%
\begin{equation*}
\begin{array}{ccc}
D_{1}\times _{C}D_{2} & \overset{p_{2}}{\rightarrow } & D_{2} \\ 
p_{1}\downarrow &  & \downarrow f_{2} \\ 
D_{1} & \underset{f_{1}}{\rightarrow } & C%
\end{array}%
\end{equation*}%
called the \emph{fiber product} of $f_{1}$ and $f_{2}$.

A morphism in $\mathcal{C}$ is called a \emph{regular epimorphism} if it is
the coequalizer of a parallel pair of arrows \cite[Definition~1.5]%
{gran_introduction_2021}. In a finitely complete category, it is then the
coequalizer of its kernel pair \cite[Exercise~1.9]{gran_introduction_2021}.

\subsection{Strong epimorphisms}

An arrow $f:A\rightarrow B$ in a category $\mathcal{C}$ is called a \emph{%
strong epimorphism} if for every commutative square%
\begin{equation*}
\begin{array}{ccc}
A & \overset{f}{\rightarrow } & B \\ 
g\downarrow &  & \downarrow h \\ 
C & \underset{m}{\rightarrow } & D%
\end{array}%
\end{equation*}%
in $\mathcal{C}$ where $m$ is monic, there exists a unique arrow $%
t:B\rightarrow C$ such that $mt=h$ and $tf=g$ \cite[Definition~1.1]%
{gran_introduction_2021}. Then in a finitely complete category, one has that
every split epimorphism is a regular epimorphism, every regular epimorphism
is a strong epimorphism, and every strong epimorphism is an epimorphism \cite%
[Proposition~1.8]{gran_introduction_2021}. Generally, an arrow is an \emph{%
isomorphism} if and only if it is a monomorphism and a strong epimorphism 
\cite[Lemma~1.3]{gran_introduction_2021}.

\subsection{Regular categories}

The notion of regular category is defined in terms of the notion of regular
epimorphism; see \cite[Definition~1.10]{gran_introduction_2021}.

\begin{definition}
Let $\mathcal{C}$ be a finitely complete category. Then $\mathcal{C}$ is 
\emph{regular} if:

\begin{enumerate}
\item the kernel pair of any morphism admits a coequalizer;

\item the pullback along any morphism of a regular epimorphism is a regular
epimorphism.
\end{enumerate}
\end{definition}

In a regular category, every arrow $f$ has an essentially unique
factorization%
\begin{equation*}
f=mq
\end{equation*}%
where $q$ is a regular epimorphism (the \emph{coimage} of $f$) and $m$ is a
monomorphism (the \emph{image} of $f$) \cite[Theorem~1.11]%
{gran_introduction_2021}. Furthermore, such a factorization is \emph{%
pullback-stable}, and these properties characterize regular categories among
finitely complete categories \cite[Theorem~1.14]{gran_introduction_2021}. A
regular category satisfies the following properties; see \cite[%
Proposition~1.13]{gran_introduction_2021}.

\begin{lemma}
\label{Lemma:regular} Let $\mathcal{C}$ be a regular category. Then:

\begin{enumerate}
\item a morphism is a regular epimorphism if and only if it is a strong
epimorphism;

\item if $gf$ is a regular epimorphism, then so is $g$;

\item if $g,f$ are regular epimorphisms, then so is $gf$ whenever defined;

\item if $g:A\rightarrow B$ and $g^{\prime }:A^{\prime }\rightarrow
B^{\prime }$ are regular epimorphisms, then so is $g\times g^{\prime
}:A\times A^{\prime }\rightarrow B\times B^{\prime }$.
\end{enumerate}
\end{lemma}

\subsection{Relations}

We now recall the notion of an (internal) \emph{relation} in a finitely
complete category $\mathcal{C}$. Let $X$ and $Y$ be objects of $\mathcal{C}$%
. A relation from $X$ to $Y$ is a pair of arrows $r_{1}:R\rightarrow X$ and $%
r_{2}:R\rightarrow Y$ that is jointly monic, i.e., such that the arrow $%
\left( r_{1},r_{2}\right) :R\rightarrow X\times Y$ is monic \cite[%
Definition~2.1]{gran_introduction_2021}. Two relations are identified when
they define the same subobject of $X\times Y$. A relation $R$ on $X$ is:

\begin{itemize}
\item \emph{reflexive} if there is an arrow $\delta :X\rightarrow R$ such
that $r_{i}\delta =1_{X}$ for $i\in \left\{ 1,2\right\} $;

\item \emph{symmetric} if there is an arrow $\sigma :R\rightarrow R$ such
that $r_{1}\sigma =r_{2}$ and $r_{2}\sigma =r_{1}$;

\item \emph{transitive} if, writing $q_{1},q_{2}:R\times_{X}R\rightarrow R$
for the pullback of $r_{2}$ and $r_{1}$, there is an arrow $%
\tau:R\times_{X}R\rightarrow R$ such that 
\begin{equation*}
r_{1}\tau=r_{1}q_{1}\qquad\text{and}\qquad r_{2}\tau=r_{2}q_{2};
\end{equation*}

\item an (internal) \emph{equivalence relation}, or a \emph{congruence}, if
it is symmetric, reflexive, and transitive.
\end{itemize}

\subsection{Barr-exact categories}

Let $\mathcal{C}$ be a finitely complete category. The \emph{kernel pair} of
an arrow $f:D\rightarrow C$ is, in particular, an (internal) equivalence
relation \cite[Lemma~2.2]{gran_introduction_2021}. Indeed, we have a
pullback diagram%
\begin{equation*}
\begin{array}{ccc}
D\times _{C}D & \overset{p_{2}}{\rightarrow } & D \\ 
p_{1}\downarrow &  & \downarrow f \\ 
D & \underset{f}{\rightarrow } & C%
\end{array}%
\end{equation*}%
Thus $R:=D\times _{C}D$, with the jointly monic pair $(p_{1},p_{2})$, is an
internal equivalence relation on $D$. An internal equivalence relation $%
R\hookrightarrow X\times X$ is \emph{effective} if it is the kernel pair of
some morphism. A regular category is \emph{Barr-exact} if every internal
equivalence relation is effective \cite[Definition~5]{carboni_regular_1998}.
In the rest of the paper, we will simply call Barr-exact categories \emph{%
exact}.

\subsection{Fibrations}

Let $F:\mathcal{F}\rightarrow \mathcal{E}$ be a functor. Fix an object $I$
of $\mathcal{E}$. Then the \emph{fiber} of $F$ at $I$ is the subcategory $%
\mathcal{F}_{I}$ of $\mathcal{F}$ comprising the objects $X$ of $\mathcal{F}$
with $F\left( X\right) =I$ and the morphisms $\alpha :X\rightarrow Y$ of $%
\mathcal{F}$ with $F\left( \alpha \right) =\mathrm{id}_{I}$ \cite[%
Definition~A.7.1]{borceux_malcev_2004}. If $\alpha :J\rightarrow I$ is a
morphism in $\mathcal{E}$ then an arrow $f:Y\rightarrow X$ in $\mathcal{F}$
is \emph{cartesian} over $\alpha $ when $F\left( f\right) =\alpha $, and
whenever $g:Z\rightarrow X$ is a morphism in $\mathcal{F}$ such that $%
F\left( g\right) $ factors as $\alpha \beta $ there exists a unique morphism 
$h:Z\rightarrow Y $ in $\mathcal{F}$ such that $F\left( h\right) =\beta $
and $g=fh$ \cite[Definition A.7.2]{borceux_malcev_2004}.

Let $\mathcal{E}$ and $\mathcal{F}$ be categories, and $F:\mathcal{F}%
\rightarrow \mathcal{E}$ a functor. Then $F$ is a \emph{fibration} when for
every arrow $\alpha :J\rightarrow I$ in $\mathcal{E}$ and every object $X$
in the fiber over $I$ there exists in $\mathcal{F}$ a \emph{cartesian
morphism} $f:Y\rightarrow X$ over $\alpha $ \cite[Definition A.7.3]%
{borceux_malcev_2004}.

Let $\mathcal{E}$ be a category with pullbacks of split epimorphisms. The
category $\mathrm{Pt}\left( \mathcal{E}\right) $ of \emph{points} of $%
\mathcal{E}$ is defined to have as objects the split epimorphisms of $%
\mathcal{E}$ with a given splitting, and as morphisms the pairs of morphisms
in $\mathcal{E}$ that commute with both the split epimorphism and its
splitting. The \emph{fibration of points} of $\mathcal{E}$ is the functor $%
\mathrm{Pt}\left( \mathcal{E}\right) \rightarrow \mathcal{E}$ that assigns
to an object of $\mathrm{Pt}\left( \mathcal{E}\right) $ corresponding to a
split epimorphism $p$ with section $s$ the target of $p$ \cite[Theorem 2.1.15%
]{borceux_malcev_2004}. The fiber over $I\in \mathcal{E}$ of the functor $%
\mathrm{Pt}\left( \mathcal{E}\right) \rightarrow \mathcal{E}$ is the
category $\mathrm{Pt}_{I}\left( \mathcal{E}\right) $ of points over $I$. A
morphism in $\mathrm{Pt}\left( \mathcal{E}\right) $ is cartesian precisely
when the corresponding commuting square with the split epimorphisms as sides
is a pullback.

\subsection{Split Short Five Lemma}

Consider, in a category, a diagram%
\begin{equation*}
\begin{array}{lllll}
K & \overset{\kappa ^{\prime }}{\rightarrow } & A^{\prime } & 
\rightleftarrows _{\beta ^{\prime }}^{\alpha ^{\prime }} & B \\ 
\downarrow \mathrm{id}_{K} &  & \downarrow \theta &  & \downarrow \mathrm{id}%
_{B} \\ 
K & \overset{\kappa }{\rightarrow } & A & \rightleftarrows _{\beta }^{\alpha
} & B%
\end{array}%
\end{equation*}%
where:

\begin{itemize}
\item $\alpha $ and $\alpha ^{\prime }$ are split epimorphisms with
splittings $\beta $ and $\beta ^{\prime }$, i.e., $\alpha \beta =1_{B}$ and $%
\alpha ^{\prime }\beta ^{\prime }=1_{B}$;

\item $\kappa $ and $\kappa ^{\prime }$ are kernels of $\alpha $ and $\alpha
^{\prime }$;

\item the diagram \emph{commutes}, i.e., $\theta \kappa ^{\prime }=\kappa $, 
$\alpha \theta =\alpha ^{\prime }$, and $\theta \beta ^{\prime }=\beta $.
\end{itemize}

The category satisfies the Split Short Five Lemma if, in any such diagram, $%
\theta $ must be an isomorphism; see \cite[Definition 3.1.1]%
{borceux_malcev_2004}.

A category $\mathcal{E}$ is \emph{pointed} if it admits an object $\ast $
that is both initial and terminal. Thus, for every object $A$ of $\mathcal{E}
$ there exists a unique arrow $\ast \rightarrow A$ and a unique arrow $%
A\rightarrow \ast $; see \cite[Definition 0.2.1]{borceux_malcev_2004}. More
generally, for every pair of objects $A,B$ there exists a unique arrow $\ast
:A\rightarrow B$ that factors through $\ast $. The following is \cite[%
Proposition 3.1.2]{borceux_malcev_2004}:

\begin{proposition}
If $\mathcal{E}$ is a pointed category with pullbacks of split epimorphisms,
the following conditions are equivalent:

\begin{enumerate}
\item the Split Short Five Lemma holds;

\item for every object $Y$ of $\mathcal{E}$ and corresponding morphism $%
\alpha _{Y}:\ast \rightarrow Y$, the inverse image functor%
\begin{equation*}
\alpha _{Y}^{\ast }:\mathrm{Pt}_{Y}\left( \mathcal{E}\right) \rightarrow 
\mathrm{Pt}_{\ast }\left( \mathcal{E}\right)
\end{equation*}%
of the fibration of points of $\mathcal{E}$ reflects isomorphisms;

\item in the fibration of points of $\mathcal{E}$, all inverse image
functors reflect isomorphisms.
\end{enumerate}
\end{proposition}

\subsection{Protomodular and homological categories}

We recall the notion of a (Bourn) \emph{protomodular category} \cite[%
Definition 3.1.3]{borceux_malcev_2004}:

\begin{definition}
\label{Definition:protomodular} A category $\mathcal{E}$ is \emph{%
protomodular} when:

\begin{enumerate}
\item $\mathcal{E}$ has pullbacks of split epimorphisms along any map;

\item all the inverse image functors of the fibration $\pi :\mathrm{Pt}%
\left( \mathcal{E}\right) \rightarrow \mathcal{E}$ of points reflect
isomorphisms.
\end{enumerate}
\end{definition}

For example, the category of groups is protomodular \cite[Example 3.1.4]%
{borceux_malcev_2004}. Every abelian category is protomodular \cite[%
Example~3.1.5]{borceux_malcev_2004}.

In terms of pointed categories and protomodular categories, one defines 
\emph{homological categories} \cite[Definition 4.1.1]{borceux_malcev_2004}:

\begin{definition}
\label{Definition:homological} A category $\mathcal{E}$ is \emph{homological}
when:

\begin{enumerate}
\item $\mathcal{E}$ is \emph{pointed};

\item $\mathcal{E}$ is \emph{regular};

\item $\mathcal{E}$ is \emph{protomodular}.
\end{enumerate}
\end{definition}

\subsection{Subcategories}

The natural notion of morphism between regular categories is given by \emph{%
regular functors}. These are the functors that preserve finite limits and
regular epimorphisms.

\begin{definition}
\label{Definition:regular-subcategory} Suppose that $\mathcal{C}$ is a
regular category. Define a regular subcategory of $\mathcal{C}$ to be a full
subcategory $\mathcal{D}$ of $\mathcal{C}$ that is also a regular category,
such that the inclusion functor $\mathcal{D}\rightarrow \mathcal{C}$ is
regular.
\end{definition}

It is easily seen that if $\mathcal{C}$ is a protomodular category and $%
\mathcal{D}$ is a regular subcategory of $\mathcal{C}$, then $\mathcal{D}$
is also protomodular. In a similar fashion one defines the notion of exact
subcategory.

\begin{definition}
\label{Definition:exact-subcategory} Suppose that $\mathcal{C}$ is an exact
category. Define an exact subcategory of $\mathcal{C}$ to be a regular
subcategory that is also exact.
\end{definition}

\subsection{Normal monics}

Let $\mathcal{E}$ be a category with finite limits. Then a (necessarily
monic \cite[Lemma 3.2.2]{borceux_malcev_2004}) morphism $f:X\rightarrow Y$
is \emph{normal} to some equivalence relation $r:R\rightarrow Y\times Y$ on $%
Y$ \cite[Definition 3.2.1]{borceux_malcev_2004} when:

\begin{enumerate}
\item $f\times f$ factors through $r$ yielding a pullback%
\begin{equation*}
\begin{array}{ccc}
X\times X & \overset{\xi }{\rightarrow } & R \\ 
\downarrow &  & \downarrow r \\ 
X\times X & \underset{f\times f}{\rightarrow } & Y\times Y%
\end{array}%
\end{equation*}

\item the diagram%
\begin{equation*}
\begin{array}{ccc}
X\times X & \overset{\xi }{\rightarrow } & R \\ 
p_{0}\downarrow &  & \downarrow d_{0} \\ 
X & \underset{f}{\rightarrow } & Y%
\end{array}%
\end{equation*}%
is also a pullback; see also \cite{bourn_normal_2000,bourn_direct_2005}.
\end{enumerate}

A monic arrow $f$ is normal if it is normal to some equivalence relation $R$%
, which is uniquely determined by $f$ when $\mathcal{E}$ is protomodular
with finite limits \cite[Theorem 3.2.8]{borceux_malcev_2004}.

\subsection{Exact homological categories}

In this paper, we are interested in homological categories and \emph{exact}
homological categories. Exact homological categories fall just short of
being \emph{semi-abelian}, which would additionally require them to have
binary coproducts; see \cite[Section 2.5]{janelidze_semi-abelian_2002}.

A category $\mathcal{C}$ is \emph{homological} if and only if it satisfies
the following conditions:

\begin{enumerate}
\item it has binary products and a zero object;

\item it has pullbacks of (split) epimorphisms;

\item it has coequalizers of kernel pairs;

\item it satisfies the Split Short Five Lemma;

\item regular epimorphisms are pullback-stable.
\end{enumerate}

See \cite[Section 2.5]{janelidze_semi-abelian_2002} and \cite[Chapters 3
and~4]{borceux_malcev_2004}. It is also exact if, in addition, all
equivalence relations are effective.

A homological category is exact if and only if every normal arrow is a
kernel (the converse being always true); see \cite%
{metere_bourn-normal_2017,bourn_normal_2000,bourn_direct_2005}.

Exact homological categories have been studied in \cite%
{bourn_direct_2005,everaert_baer_2004,everaert_baer_2004-1,gray_relative_2016,gray_complete_2022,bourn_protomodularity_1998}%
.

\subsection{The exact completion of a regular category}

A regular category $\mathcal{C}$ admits a canonical completion to an exact
category \textrm{Ex}$(\mathcal{C})$, containing $\mathcal{C}$ as a regular
subcategory, such that, for every regular functor $F:\mathcal{C}\rightarrow 
\mathcal{E}$ where $\mathcal{E}$ is an exact category, there exists an
essentially unique \emph{regular} functor $\hat{F}:\mathrm{Ex}\left( 
\mathcal{C}\right) \rightarrow \mathcal{E}$ such that $\hat{F}|_{\mathcal{C}%
}\cong F$. Such a completion is also called the ex/reg completion, and
denoted by \textrm{Ex}$_{\mathrm{reg}}\left( \mathcal{C}\right) $ or $%
\mathcal{C}_{\mathrm{ex/reg}}$, to stress the fact that it is the exact
completion of $\mathcal{C}$ \emph{as a regular category}. Different notions
of exact completion exist for more general categories, such as categories
with finite limits \cite{carboni_free_1982}. Naturally, these completions
generally do \emph{not} coincide when applied to a regular category.

One can explicitly construct the exact completion of a regular category $%
\mathcal{C}$ as follows. Let $\mathrm{Rel}\left( \mathcal{C}\right) $ be the
category that has the same objects as $\mathcal{C}$. A morphism $%
X\rightarrow Y$ in $\mathrm{Rel}\left( \mathcal{C}\right) $ is a relation,
namely a jointly monic span 
\begin{equation*}
X\xleftarrow{\ r_{1}\ }R\xrightarrow{\ r_{2}\ }Y,
\end{equation*}%
or equivalently a monomorphism $\left( r_{1},r_{2}\right) :R\rightarrow
X\times Y$. If $R$ is a relation from $X$ to $Y$ and $S$ is a relation from $%
Y$ to $Z$, their composite is obtained by forming the pullback $R\times
_{Y}S $, mapping it to $X\times Z$, and taking the regular image. The
identity relation on $X$ is the diagonal $\Delta _{X}$. Given a morphism $%
f:X\rightarrow Y$ in $\mathcal{C}$, one can regard it as a morphism of $%
\mathrm{Rel}\left( \mathcal{C}\right) $ by identifying it with its graph.
Furthermore, any relation $R$ from $X$ to $Y$ has a corresponding \emph{%
opposite} relation $R^{\circ }$ from $Y$ to $X$. Then, if $f:X\rightarrow Y$
is an arrow in $\mathcal{C}$, the relation $f^{\circ }f$ is the kernel pair
of $f$, and $ff^{\circ }=\Delta _{Y}$ if and only if $f$ is a regular
epimorphism \cite{carboni_diagram_1991}.

One can describe the exact completion $\mathrm{Ex}\left( \mathcal{C}\right) $
in terms of \textrm{Rel}$\left( \mathcal{C}\right) $ as follows. The objects
of $\mathrm{Ex}\left( \mathcal{C}\right) $ are the \emph{equivalence
relations} $R$ on objects $X$ of $\mathcal{C}$. If $R$ and $S$ are
equivalence relations on $X$ and $Y$, respectively, a morphism $\left(
X,R\right) \rightarrow \left( Y,S\right) $ is a relation $T$ from $X$ to $Y$
satisfying 
\begin{equation*}
TR=T=ST,\qquad R\leq T^{\circ }T,\qquad TT^{\circ }\leq S.
\end{equation*}%
The first two equalities express compatibility and saturation, while the two
inequalities express totality and single-valuedness modulo the given
equivalence relations. Composition is inherited from $\mathrm{Rel}\left( 
\mathcal{C}\right) $, and the identity morphism on $\left( X,R\right) $ is $%
R $. The canonical regular embedding of $\mathcal{C}$ into its exact
completion sends $X$ to $\left( X,\Delta _{X}\right) $ and a morphism to its
graph; see \cite[Definition 11]{carboni_regular_1998}.

Suppose now that $\mathcal{C}$ is a homological category. Then, in
particular, $\mathcal{C}$ is a regular category. Thus, one can consider the
exact completion $\mathrm{Ex}\left( \mathcal{C}\right) $ of $\mathcal{C}$.
Then, by definition, $\mathrm{Ex}\left( \mathcal{C}\right) $ is exact.
Furthermore, the hypothesis that $\mathcal{C}$ is a homological category
implies that $\mathrm{Ex}\left( \mathcal{C}\right) $ is also a homological
category \cite[Proposition 5.2]{gran_semi-localizations_2016}. Thus, in this
case $\mathrm{Ex}\left( \mathcal{C}\right) $ is an \emph{exact homological
category}.

The relation-theoretic construction of the exact completion of a regular
category goes back to \cite{succi_cruciani_teoria_1975}, while \cite%
{carboni_regular_1998} established the general existence and universal
properties of regular and exact completions. Infinitary and weak-limit
variants were studied in \cite{hu_note_1996}, and \cite{lack_note_1999} gave
a sheaf-theoretic construction of the ex/reg completion. Further
perspectives can be found in \cite%
{gran_semi-localizations_2016,garner_lex_2012,shulman_exact_2012}.

\subsection{Exact sequences}

Let $\mathcal{C}$ be a homological category. One defines a \emph{short exact
sequence}%
\begin{equation*}
\ast {}\longrightarrow K\overset{k}{\longrightarrow }X\overset{p}{%
\longrightarrow }Y\longrightarrow \ast
\end{equation*}%
to be a pair $\left( k,p\right) $ where $p$ is a \emph{regular epimorphism}
and $k$ is the kernel of $p$ (i.e., the pullback of $p$ along the trivial
map $\ast \rightarrow Y$); see also \cite%
{bourn_3_2001,bourn_torsion_2006,clementino_torsion_2006,gran_torsion_2007,bourn_baer_2007,janelidze_characterization_2007}%
.

The following definition extends the usual notion of \emph{thick} (or \emph{%
Serre}) subcategory \cite[Definition 8.3.21]{kashiwara_categories_2006} to
arbitrary homological categories.

\begin{definition}
\label{Definition:thick} Let $\mathcal{C}$ be a homological category. A 
\emph{thick} subcategory $\mathcal{S}$ of $\mathcal{C}$ is a regular full
subcategory such that, for every subobject $X$ of an object of $\mathcal{S}$%
, and for every short exact sequence%
\begin{equation*}
\ast {}\longrightarrow K\overset{k}{\longrightarrow }X\overset{p}{%
\longrightarrow }Y\longrightarrow \ast
\end{equation*}%
in $\mathcal{C}$, $X$ is isomorphic to an object of $\mathcal{S}$ if and
only if both $K$ and $Y$ are isomorphic to objects of $\mathcal{S}$.
\end{definition}

Related notions of hereditary torsion class and Birkhoff subcategory have
been considered in \cite%
{bourn_torsion_2006,clementino_torsion_2006,lopez_cafaggi_torsion_2022}. The
proof of \cite[Lemma 2.13]{lopez_cafaggi_torsion_2022} shows that if $%
\mathcal{S}$ is a thick subcategory of a homological category $\mathcal{C}$
and $\mathcal{C}$ is exact, then $\mathcal{S}$ is also exact. (See also \cite%
{everaert_baer_2004,janelidze_galois_1994} for the same conclusion under the
stronger hypothesis that $\mathcal{S}$ is a Birkhoff subcategory.)

A proof similar to that of \cite[Proposition 2.2]%
{gran_semi-localizations_2016}, where a related result is stated under the
stronger assumption that $\mathcal{S}$ is closed under subobjects, gives the
following proposition.

\begin{proposition}
\label{Proposition:completion-thick} Let $\mathcal{C}$ be a homological
category, and $\mathcal{S}$ a thick subcategory of $\mathcal{C}$ in the
sense of \Cref{Definition:thick}. Then the universal property of $\mathrm{Ex}%
\left( \mathcal{S}\right) $ applied to the composition $\mathcal{S}%
\rightarrow \mathcal{C}\rightarrow \mathrm{Ex}\left( \mathcal{C}\right) $
produces a functor $\mathrm{Ex}\left( \mathcal{S}\right) \rightarrow \mathrm{%
Ex}\left( \mathcal{C}\right) $ that establishes an equivalence between $%
\mathrm{Ex}\left( \mathcal{S}\right) $ and an exact subcategory of $\mathrm{%
Ex}\left( \mathcal{C}\right) $.
\end{proposition}

\subsection{The abelian case}

\label{sec:abelian-case}

Let $\mathcal{A}$ be a quasi-abelian category and let $\func{LH}(\mathcal{A}%
) $ be its left heart in the sense of \cite{schneiders_quasi-abelian_1999};
see also \cite{henrard_left_2022}. A quasi-abelian category is a homological
category, which is exact if and only if it is an abelian category. The
canonical embedding 
\begin{equation*}
\mathcal{A}\longrightarrow \func{LH}(\mathcal{A})
\end{equation*}%
is fully faithful and regular, its image is closed under subobjects, and
every object of the left heart is a quotient of an object of $\mathcal{A}$ 
\cite[Theorem~1.3]{henrard_left_2022}. By the recognition criterion for
exact completions \cite{carboni_regular_1998}, this allows one to identify $%
\mathrm{LH}\left( \mathcal{A}\right) $ with the exact completion of $%
\mathcal{A}$ as a regular category.

\section{Exact completion of Polish groups\label{Section:completion}}

In this section we observe that Polish groups form a \emph{homological
category}, and describe its exact completion (as a regular category) as the
category of \emph{groups with a Polish cover}.

\subsection{Polish groups}

We let $\mathbf{P}$ be the category whose objects are Polish groups and
whose arrows are continuous group homomorphisms. Its full subcategory $%
\mathbf{PAb}$ spanned by Polish abelian groups was considered in \cite%
{lupini_looking_2024}. As observed there, $\mathbf{PAb}$ is a quasi-abelian
category in the sense of \cite{schneiders_quasi-abelian_1999}.

\begin{proposition}
\label{prop:regularity} The category $\mathbf{P}$ is regular and pointed.
Its regular epimorphisms are the continuous surjective homomorphisms.
\end{proposition}

\begin{proof}
The trivial group is terminal and initial, finite products carry the product
topology, and the equalizer of two continuous homomorphisms is a closed
subgroup. Hence $\mathbf{P}$ has finite limits.

Let $f:G\rightarrow H$ be a continuous homomorphism and put $N=\ker (f)$.
The kernel pair $G\times _{H}G\rightrightarrows G$ is the coset relation of
the closed normal subgroup $N$, and its coequalizer is the quotient map $%
G\rightarrow G/N$. If $f$ is surjective, the Open Mapping Theorem for Polish
groups identifies $H$ topologically with $G/N$; hence $f$ is a regular
epimorphism. Conversely, coequalizers in $\mathbf{P}$ are quotient
homomorphisms and are therefore surjective.

Finally, the pullback of a surjective homomorphism along an arbitrary
continuous homomorphism is again surjective. Its domain is a closed subgroup
of a product of Polish groups and is therefore Polish. Thus regular
epimorphisms are pullback-stable.
\end{proof}

The subobjects of a Polish group in $\mathbf{P}$ can be identified with its
Polish subgroups. Likewise, it is easily seen that the internal equivalence
relations in $\mathbf{P}$ correspond up to isomorphism to groups with a
Polish cover, and kernels correspond to closed normal subgroups. It easily
follows that $\mathbf{P}$ is a regular category, which is also a consequence
of the results in \cite{borceux_topological_2005}.

The Split Short Five Lemma is also verified directly for $\mathbf{P}$, which
proves protomodularity. This can actually be regarded as a special instance
of general results pertaining to internal groups in a finitely complete
category; see \cite{bourn_regular_2004,borceux_malcev_2004}. We conclude
that $\mathbf{P}$ is indeed a \emph{homological category}.

\subsection{Subcategories}

A short exact sequence in $\mathbf{P}$ is a short exact sequence of groups%
\begin{equation}
1\longrightarrow N\longrightarrow G\longrightarrow H\longrightarrow 1
\label{Eq:extension}
\end{equation}%
where all the homomorphisms are continuous. Thus, $G\rightarrow H$ is a
surjective continuous group homomorphism with kernel (the image of) $N$.
Consequently, a full replete subcategory $\mathcal{S}$ of $\mathbf{P}$ is a
regular subcategory whenever it is closed under finite products, closed
subgroups, and quotients by closed normal subgroups. Such a subcategory is
thick if and only if, for any extension as in \eqref{Eq:extension} such that 
$H$ and $N$ are in $\mathcal{S}$, if $G$ is a subobject of an element of $%
\mathcal{S}$, then $G$ is in $\mathcal{S}$.

We now consider the classes of Polish groups introduced in %
\Cref{Subsection:classes}.

\begin{proposition}
The following are \emph{thick subcategories} of the category of Polish
groups:

\begin{enumerate}
\item $\mathbf{PAb}$;

\item $\mathbf{nAP}$;

\item $\mathbf{CLI}$;

\item $\mathbf{LCP}$;

\item $\mathbf{Lie}$;

\item $\mathbf{proLiePAb}$;

\item $\mathbf{CP}$;

\item $\mathbf{AmenLCP}$;

\item $\mathbf{EnALCP}$;

\item $\mathbf{SolP}$;

\item $\mathbf{pro\text{-}\mathcal{F}P}$, for every extension-closed variety 
$\mathcal{F}$ of finite groups;

\item $\mathbf{FDLCP}$;

\item $\mathbf{LCPAb}_{\mathrm{cg}}$;

\item $\mathbf{TorLCPAb}$;

\item $\mathbf{LCPAb}(p)$, for every prime $p$;

\item $\mathbf{FLCPAb}$.
\end{enumerate}
\end{proposition}

\begin{proof}
(1) The category $\mathbf{PAb}$ is a regular subcategory of $\mathbf{P}$; in
fact it is quasi-abelian \cite{lupini_looking_2024}. If $X$ is a subobject
of a Polish abelian group, then $X$ is algebraically abelian. Hence, in
every short exact sequence with middle term $X$, both its kernel and
quotient are abelian. Thus both sides of the equivalence in Definition~\ref%
{Definition:thick} hold.

(2) Non-Archimedean Polish groups are closed under closed subgroups and
quotients by closed normal subgroups; see \cite[Section~2.4]%
{gao_invariant_2009}. They are furthermore closed under extensions by \cite[%
Exercise~2.4.1]{gao_invariant_2009}.

(3) The assertion follows from the CLI three-space theorem: if $N$ is a
closed normal subgroup of a Polish group $G$, then $G$ is CLI if and only if
both $N$ and $G/N$ are CLI \cite[Theorem~3.C.1]{becker_polish_1998}.

(4) Local compactness is a three-space property for topological groups; see
for example \cite[Section~3]{serre_extensions_1952}.

(5) This follows from standard results; see \cite[Th\'{e}or\`{e}me~3]%
{serre_extensions_1952}.

(6) Pro-Lie Polish abelian groups are closed under the operations defining a
regular subcategory and under extensions \cite[Theorem~3.17]%
{casarosa_homological_2026}. Since every subobject of an abelian group is
abelian, the cited thickness result in $\mathbf{PAb}$ also gives thickness
as a subcategory of $\mathbf{P}$ in the sense of Definition~\ref%
{Definition:thick}.

(7) Compact groups are closed under finite products, closed subgroups,
quotients by closed normal subgroups, and extensions; see, for example, \cite%
{hofmann_structure_2013}.

(8) Among locally compact groups, amenability is inherited by closed
subgroups and quotients and is a three-space property: if $N$ is a closed
normal subgroup of a locally compact group $G$, then $G$ is amenable if and
only if both $N$ and $G/N$ are amenable \cite[Lemma~5.5]%
{reid_wesolek_dense_2018}; see also \cite{paterson_amenability_1988}.
Together with (4), this proves that $\mathbf{AmenLCP}$ is thick.

(9) The class of elementary non-Archimedean (equivalently, totally
disconnected) locally compact second-countable groups is closed under closed
subgroups, quotients by closed normal subgroups, and extensions \cite%
{wesolek_elementary_2015}. Therefore $\mathbf{EnALCP}$ is thick.

(10) Subgroups and quotients of solvable groups are solvable, and an
extension of a solvable group by a solvable group is solvable; see, for
example, \cite{robinson_course_1982}. These algebraic facts imply all the
required closure properties for $\mathbf{SolP}$.

(11) Let $\mathcal{F}$ be an extension-closed variety of finite groups. Pro-$%
\mathcal{F}$ groups are closed under products, closed subgroups, quotients
by closed normal subgroups, and extensions \cite[Proposition~2.2.1]%
{ribes_profinite_2010}. Restricting to the compact Polish groups proves that 
$\mathbf{pro\text{-}\mathcal{F}P}$ is thick.

(12) It follows from the formula pertaining to dimensions of extensions from 
\cite[Theorem~2.1]{nagami_dimension-theoretical_1962}.

(13)--(16) These are thick subcategories of the quasi-abelian category $%
\mathbf{PAb}$ by \cite[Theorem~6.18(3), (9)--(12)]{lupini_applications_2025}%
. The original permanence results used there include \cite[Theorem~2.6(2)]%
{moskowitz_homological_1967} for compactly generated groups, \cite[%
Chapters~2--3]{armacost_structure_1981} for topological torsion and
topological $p$-groups, and \cite[Proposition~2.9]{hoffmann_homological_2007}
for groups of finite ranks. The same argument as in (6) shows that they are
thick as subcategories of $\mathbf{P}$ in the sense of Definition~\ref%
{Definition:thick}: a subobject in $\mathbf{P}$ of an abelian group is
abelian, so every relevant short exact sequence is a short exact sequence in 
$\mathbf{PAb}$.
\end{proof}

\begin{proposition}
\label{Proposition:non-thick}The following are \emph{regular subcategories}
of the category of Polish groups that are not thick:

\begin{enumerate}
\item $\mathbf{TSI}$;

\item $\mathbf{nATSI}$;

\item $\mathbf{proLieP}$;

\item $\mathbf{proLieTSI}$;

\item $\mathbf{Hilb}$, regarded as a subcategory of $\mathbf{PAb}$ and hence
of $\mathbf{P}$.

\item $\mathbf{Ban}$, regarded as a subcategory of $\mathbf{PAb}$ and hence $%
\mathbf{P}$.
\end{enumerate}
\end{proposition}

\begin{proof}
The closure properties needed for regularity are standard. For $\mathbf{TSI}$
it follows from standard results pertaining to compatible metrics on Polish
groups \cite[Section~2.2]{gao_invariant_2009}. The Closed Subgroup Theorem
and stability under products give the corresponding closure properties for
pro-Lie groups \cite[Chapter~3]{hofmann_lie_2007}. If $N$ is a closed normal
subgroup of a Polish pro-Lie group $G$, then $G/N$ is Polish, hence
complete; therefore the Quotient Theorem implies that $G/N$ is a pro-Lie
group \cite[Theorem~4.1]{hofmann_lie_2007}. Finite products and kernels of
bounded operators between separable real Hilbert spaces are again separable
real Hilbert spaces. Moreover, a surjective bounded operator is open and its
quotient is a Hilbert space. Thus $\mathbf{Hilb}$ is regular and its
inclusion in $\mathbf{PAb}$ is regular. The same argument, using the open
mapping theorem, shows that $\mathbf{Ban}$ and its inclusion in $\mathbf{PAb}
$ are regular.

(1)--(4) We give one counterexample to thickness that applies to the classes
(1), (2), (3), and (4) simultaneously. Fix a prime $p$, and consider the
non-Archimedean Banach space over $\mathbb{Q}_{p}$ 
\begin{equation*}
B=c_{0}(\mathbb{Z},\mathbb{Q}_{p}),\qquad \lVert x\rVert _{\infty
}=\sup_{n\in \mathbb{Z}}\lvert x_{n}\rvert _{p}.
\end{equation*}%
Define also the Banach subspace 
\begin{equation*}
K=\left\{ x\in B:\lim_{\lvert n\rvert \rightarrow \infty }p^{\lvert n\rvert
}\lvert x_{n}\rvert _{p}=0\right\},\qquad \lVert x\rVert _{K}=\sup_{n\in 
\mathbb{Z}}p^{\lvert n\rvert }\lvert x_{n}\rvert _{p}.
\end{equation*}%
Then $K$, endowed with $\lVert \cdot \rVert _{K}$, is a non-Archimedean TSI
Polish group, and the inclusion $K\rightarrow B$ is continuous and
injective. Let $\sigma $ be the bilateral shift on $B$. It is an isometry of 
$B$ and a continuous automorphism of $K$. Consequently, 
\begin{equation*}
A:=B\rtimes _{\sigma }\mathbb{Z}
\end{equation*}%
is a non-Archimedean TSI Polish group: the products of the $\lVert \cdot
\rVert _{\infty }$-balls in $B$ with $\{0\}$ form a basis of open normal
subgroups of $A$. In particular, $A$, $B$, and $K$ are pro-countable by \cite%
[Theorem~1.1]{gao_non-archimedean_2014} and \cite[Proposition~2.2]%
{ding_non-archimedean_2017}, hence pro-Lie.

On the other hand, put 
\begin{equation*}
X=K\rtimes _{\sigma }\mathbb{Z}.
\end{equation*}%
The inclusion $K\rightarrow B$ induces a monomorphism $X\rightarrow A$ in $%
\mathbf{P}$, and there is a split short exact sequence 
\begin{equation*}
1\longrightarrow K\longrightarrow X\longrightarrow \mathbb{Z}\longrightarrow
1.
\end{equation*}%
The group $X$ is non-Archimedean, but it is not TSI. Indeed, the family $%
(\sigma ^{m})_{m\in \mathbb{Z}}$ is not equicontinuous on $K$: if $e_{0}$ is
supported at $0$, then 
\begin{equation*}
\lVert \sigma ^{m}e_{0}\rVert _{K}=p^{m}\lVert e_{0}\rVert _{K}\qquad (m\geq
0).
\end{equation*}%
If $X$ were TSI, intersecting a basis of conjugation-invariant identity
neighborhoods of $X$ with $K$ would make this family equicontinuous, a
contradiction.

Finally, a non-Archimedean pro-Lie Polish group is necessarily TSI. Indeed,
the Pro-Lie Group Theorem provides a cofinal family of co-Lie closed normal
subgroups; the corresponding Lie quotients are non-Archimedean and hence
discrete, so these kernels form a basis of open normal subgroups. Thus $X$
is not pro-Lie. We have therefore exhibited a subobject $X$ of the object $A$
such that the kernel and quotient in the displayed short exact sequence
belong to each of the subcategories in (1)--(4), while $X$ does not.

(5) Kalton and Peck constructed the twisted Hilbert space $Z_{2}$ and a
short exact sequence of separable real Banach spaces 
\begin{equation*}
0\longrightarrow \ell _{2}\longrightarrow Z_{2}\longrightarrow \ell
_{2}\longrightarrow 0
\end{equation*}%
such that $Z_{2}$ is not isomorphic to a Hilbert space \cite[Theorem~6.1]%
{kalton_twisted_1979}. Since $Z_{2}$ is separable, choose a countable
norming family $(f_{n})_{n\geq 1}$ in its dual unit ball. The map 
\begin{equation*}
Z_{2}\longrightarrow \ell _{2},\qquad x\longmapsto \left(
2^{-n}f_{n}(x)\right) _{n\geq 1}
\end{equation*}%
is a continuous injective homomorphism, and hence exhibits $Z_{2}$ as a
subobject in $\mathbf{PAb}$ of an object of $\mathbf{Hilb}_{\mathbb{R}}$.
Thus the displayed short exact sequence has kernel and quotient in $\mathbf{%
Hilb}$, whereas its middle term is not. Together with the preceding example,
this is precisely the failure of Definition~\ref{Definition:thick} in the
first five cases.

(6) Consider the Kalton--Peck quasi-Banach space $Z_{1}$, obtained from the
construction in \cite[Section~4]{kalton_twisted_1979}; this notation is used
in \cite[Section~1.2]{cabello_sanchez_strictly_2012}. On the dense subspace $%
\ell _{1}\times c_{00}$ it is induced by the quasilinear map 
\begin{equation*}
\Omega _{1}(x)_{n}=x_{n}\log \frac{\lVert x\rVert _{1}}{\lvert x_{n}\rvert }
\end{equation*}%
(with value zero when $x_{n}=0$) and the quasinorm 
\begin{equation*}
\lVert (y,x)\rVert _{\Omega }=\lVert y-\Omega _{1}(x)\rVert _{1}+\lVert
x\rVert _{1}.
\end{equation*}%
Its completion fits into a short exact sequence of separable quasi-Banach
spaces 
\begin{equation*}
0\longrightarrow \ell _{1}\longrightarrow Z_{1}\overset{q}{\longrightarrow }%
\ell _{1}\longrightarrow 0,
\end{equation*}%
and the quotient map $q$ is strictly singular \cite[Section~3]%
{castillo_strictly_2002}; see also \cite[Section~1.2 and Theorem~2]%
{cabello_sanchez_strictly_2012}. Consequently, $Z_{1}$ is not locally
convex: otherwise it would be isomorphic to a Banach space, and the
surjection $q$ onto $\ell _{1}$ would admit a continuous linear right
inverse by the projectivity of $\ell _{1}$ in the Banach category \cite[%
Section~2.7]{cabello_sanchez_homological_2023}, contradicting strict
singularity. The nontriviality of the Kalton--Peck extension is also given
by \cite[Theorem~4.2(c)]{kalton_twisted_1979}; for a modern treatment, see 
\cite[Proposition~3.3.5 and Section~3.4]{cabello_sanchez_homological_2023}.
The case $p=1$ is essential here: for $1<p<\infty $ the corresponding space $%
Z_{p}$ is a Banach space \cite[Theorem~4.7]{kalton_twisted_1979}, and the
failure of local convexity as a three-space property at $\ell _{1}$ goes
back to \cite[Theorem~1]{ribe_examples_1979}. In particular, the additive
Polish group of $Z_{1}$ is not isomorphic to an object of $\mathbf{Ban}$,
since every continuous additive map between real topological vector spaces
is $\mathbb{R}$-linear.

Nevertheless, $Z_{1}$ is a subobject of an object of $\mathbf{Ban}$. Indeed,
let $u_{n}(y,x)=y_{n}$ on $\ell _{1}\times c_{00}$. The inequality $t\log
(1/t)\leq e^{-1}$ for $0\leq t\leq 1$ gives 
\begin{equation*}
\lvert u_{n}(y,x)\rvert \leq \lVert y-\Omega _{1}(x)\rVert _{1}+e^{-1}\lVert
x\rVert _{1}\leq \lVert (y,x)\rVert _{\Omega }.
\end{equation*}%
Thus each $u_{n}$ extends continuously to $Z_{1}$, and 
\begin{equation*}
J:Z_{1}\longrightarrow \ell _{1}\oplus _{1}\ell _{1},\qquad Jz=\left(
qz,\left( 2^{-n}u_{n}(z)\right) _{n\geq 1}\right) ,
\end{equation*}%
is continuous. It is injective because $qz=0$ places $z$ in the canonical
copy of $\ell _{1}$, on which the $u_{n}$ are the usual coordinate
functionals. Hence $Z_{1}$ is a subobject of a separable real Banach space,
while the kernel and quotient of the preceding short exact sequence are
Banach spaces and its middle term is not. This proves that $\mathbf{Ban}$ is
not thick.
\end{proof}

\subsection{Extensions}

Recall that a \emph{standard Borel group} is a group equipped with a
standard Borel structure for which multiplication and inversion are Borel;
see \cite{mackey_borel_1957}. We consider extensions 
\begin{equation}
1\longrightarrow M\longrightarrow P\overset{\pi }{\longrightarrow }%
G\longrightarrow 1
\end{equation}%
where $M$ and $G$ are Polish groups, $P$ is a standard Borel group, and the
structure maps are Borel homomorphisms. Suppose moreover that $\pi $ has a 
\emph{Borel right inverse} $s$ with $s(1)=1$. Put, for $x,y\in G$ and $m\in
M $, 
\begin{align}
\alpha _{x}(m)& =s(x)ms(x)^{-1}, \\
c(x,y)& =s(x)s(y)s(xy)^{-1}.
\end{align}%
Thus $\alpha _{x}\in \func{Aut}(M)$ and $c(x,y)\in M$. The pair $(\alpha ,c)$
is the \emph{factor system} associated with $s$. The following result is
essentially established in \cite[Theorem, p.~73]{brown_extensions_1971}:

\begin{theorem}[Brown]
\label{thm:brown-extension} Suppose that 
\begin{equation*}
1\longrightarrow M\xrightarrow{\iota}P\xrightarrow{\pi}G\longrightarrow 1
\end{equation*}%
is an exact sequence of standard Borel groups, where $M$ and $G$ are Polish.
Suppose, in addition, that $\pi $ admits a Borel right inverse. Then $P$
admits a unique Polish group topology inducing its given Borel structure.
With respect to this topology, $\iota $ is a topological embedding and $\pi $
is a continuous open homomorphism.
\end{theorem}

\begin{proof}
Choose a normalized Borel section $s:G\rightarrow P$. The Borel bijection 
\begin{equation*}
M\times G\longrightarrow P\text{,}\qquad (m,g)\longmapsto \iota (m)s(g),
\end{equation*}%
transports multiplication to 
\begin{equation*}
(m,g)(n,h)=(m\alpha _{g}(n)c(g,h),gh)\text{,}
\end{equation*}%
where 
\begin{equation*}
\alpha _{g}(n)=s(g)ns(g)^{-1}\text{,}\qquad c(g,h)=s(g)s(h)s(gh)^{-1}\text{.}
\end{equation*}%
The maps $(g,n)\mapsto \alpha _{g}(n)$ and $c$ are Borel and satisfy the
nonabelian factor-system identities 
\begin{equation*}
\alpha _{g}\alpha _{h}=\func{Inn}_{c(g,h)}\alpha _{gh},\qquad
c(g,h)c(gh,k)=\alpha _{g}(c(h,k))c(g,hk).
\end{equation*}%
It is proved in \cite[Theorem, p.~73]{brown_extensions_1971} that every such
Borel nonabelian factor system of Polish groups is represented by a Polish
topological extension. Transporting that topology to $P$ gives existence and
the stated properties of $\iota $ and $\pi $. The rest of the assertions
follow from the automatic continuity of Borel homomorphisms on Polish groups 
\cite[Theorem~9.10]{kechris_classical_1995}.
\end{proof}

\subsection{Borel-definable homomorphisms}

We now recall the notion of Borel-definable homomorphism between groups with
a Polish cover, as initially considered in \cite%
{bergfalk_definable_2024,bergfalk_definable_2024-1}; see also \cite%
{lupini_looking_2024}.

\begin{definition}
\label{Definition:homo-Polish-cover} Let $G/N$ and $H/M$ be groups with a
Polish cover, and let $f:G/N\rightarrow H/M$ be a group homomorphism. We say
that $f$ is:

\begin{itemize}
\item an \emph{airomorphism} if it admits a lift $G\rightarrow H$ which is a
continuous group homomorphism;

\item \emph{Borel-definable} if it admits a Borel lift $G\rightarrow H$;

\item \emph{continuously definable} if it admits a continuous lift $%
G\rightarrow H$;

\item a \emph{homomorphism with a Polish cover} if its lifted graph 
\begin{equation*}
\widehat{\Gamma }(f)=\{(g,h)\in G\times H:f(gN)=hM\}
\end{equation*}%
is a Polish subgroup of $G\times H$.
\end{itemize}
\end{definition}

The prefix airo- comes from the Greek verb \textalpha\textiota\textrho%
\textomega\ (\emph{air\={o}}, to lift). The following result can be seen as
a generalization of \cite[Theorem 4.6]{lupini_looking_2024} to Polish groups
that are not necessarily abelian.

\begin{proposition}
\label{prop:borel-graph} Let $\varphi :G/N\rightarrow H/M$ be a homomorphism
between groups with a Polish cover. The following conditions are equivalent:

\begin{enumerate}
\item $\varphi $ is Borel-definable;

\item $\varphi $ is a homomorphism with a Polish cover;

\item there exist a group with a Polish cover $Z/L$, an airomorphism $\psi
:Z/L\rightarrow H/M$, and a bijective airomorphism $\rho :Z/L\rightarrow G/N$
such that $\varphi =\psi \circ \rho ^{-1}$.
\end{enumerate}
\end{proposition}

\begin{proof}
(2)$\Rightarrow $(3) Put $Z=\widehat{\Gamma }(\varphi )$, endowed with a
Polish group topology witnessing that it is Polish, and put $L=N\times M$, a
Polish normal subgroup of $Z$. The coordinate projections induce
airomorphisms 
\begin{equation*}
\rho :Z/L\longrightarrow G/N\quad \text{and}\quad \psi :Z/L\longrightarrow
H/M.
\end{equation*}%
The homomorphism $\rho $ is bijective, its lift $p_{G}:Z\rightarrow G$ is
continuous and surjective, and $\varphi =\psi \circ \rho ^{-1}$.

(3)$\Rightarrow $(1) Without loss of generality, we can assume that $\rho $
has a surjective continuous homomorphic lift $\widetilde{\rho }:Z\rightarrow
G$, and $\psi $ has a continuous homomorphic lift $\widetilde{\psi }%
:Z\rightarrow H$. The map $\widetilde{\rho }$ is open and has a Borel right
inverse $s:G\rightarrow Z$ \cite[Lemma~1]{brown_extensions_1971}. Thus $%
\widetilde{\psi }\circ s$ is a Borel lift of $\varphi $.

(1)$\Rightarrow $(2) Let $F:G\rightarrow H$ be a Borel lift of $\varphi $.
Since $M$ is a Borel subgroup of $H$, the identity 
\begin{equation*}
\widehat{\Gamma }(\varphi )=\{(g,h)\in G\times H:F(g)^{-1}h\in M\}
\end{equation*}%
shows that $\widehat{\Gamma }(\varphi )$ is a Borel subgroup of $G\times H$
and hence a standard Borel group. Moreover, 
\begin{equation*}
1\longrightarrow M\longrightarrow \widehat{\Gamma }(\varphi )%
\xrightarrow{\,p_G\,}G\longrightarrow 1
\end{equation*}%
is an exact sequence of standard Borel groups and $g\mapsto (g,F(g))$ is a
Borel section of $p_{G}$. By \Cref{thm:brown-extension}, $\widehat{\Gamma }%
(\varphi )$ admits a Polish group topology inducing its given Borel
structure. Its inclusion in $G\times H$ is a Borel homomorphism and is
therefore continuous, so $\widehat{\Gamma }(\varphi )$ is a Polish subgroup
of $G\times H$.
\end{proof}

\subsection{Continuous lifts}

We now observe that in the case of non-Archimedean Polish groups, every
Borel-definable group homomorphism is automatically continuously definable.
We first isolate the topological lemma behind the regularization argument.
The proof is essentially contained in the proof of \cite[Proposition~4.6]%
{bergfalk_definable_2024}.

\begin{lemma}
\label{lem:two-factor} Let $G$ be a non-Archimedean Polish group and let $%
C\subseteq G$ be a dense $G_{\delta }$ set. There are continuous maps $%
a,b:G\rightarrow C$ such that 
\begin{equation*}
x=a(x)b(x)\qquad (x\in G).
\end{equation*}
\end{lemma}

\begin{proof}
Write $C=\bigcap_{n}D_{n}$, where each $D_{n}$ is dense and open. In $%
G\times G$ consider 
\begin{equation*}
A_{n}=\{(x,y):xy\in D_{n}\text{ and }y^{-1}\in D_{n}\}.
\end{equation*}%
The set $A_{n}$ is open, and for every $x\in G$ its vertical section is 
\begin{equation*}
(A_{n})_{x}=x^{-1}D_{n}\cap D_{n}^{-1},
\end{equation*}%
which is dense and open in $G$.

Fix a complete compatible ultrametric $d$ on the underlying space of $G$. We 
% TODO [exposition]: the indexing in this recursion is inconsistent -- the
% construction is initialised at $\mathcal{P}_{-1}$, the inductive step is
% phrased as reaching "level $n-1$", and condition (3) mixes levels $n$ and
% $n+1$. The construction is fine; the bookkeeping should be made uniform.
recursively construct clopen partitions $\mathcal{P}_n$ of $G$ and, for each 
$P\in\mathcal{P}_n$, a nonempty open set $V_P\subseteq G$ such that:

\begin{enumerate}
\item $\mathcal{P}_{n+1}$ refines $\mathcal{P}_n$;

\item if $P\in \mathcal{P}_{n+1}$ is contained in $Q\in \mathcal{P}_{n}$,
then $\overline{V}_{P}\subseteq V_{Q}$;

\item $\func{diam}(V_{P})\leq 2^{-n}$ and 
\begin{equation*}
P\times \overline{V}_{P}\subseteq \bigcap_{k\leq n}A_{k}.
\end{equation*}
\end{enumerate}

Initialize the construction with $\mathcal{P}_{-1}=\{G\}$ and $V_G=G$.
Suppose it has reached level $n-1$. For every $x$ in a member $Q$ of $%
\mathcal{P}_{n-1}$, the set 
\begin{equation*}
V_Q\cap\bigcap_{k\leq n}(A_k)_x
\end{equation*}
is nonempty and open. Choose a point in this set and then sufficiently small
open neighborhoods of $x$ and of the chosen point so that the product of
their closures is contained in $\bigcap_{k\leq n}A_k$, with the second
closure also contained in $V_Q$. A clopen partition refining the resulting
open cover and $\mathcal{P}_{n-1}$ exists because $G$ is zero-dimensional.
Shrinking the neighborhoods in the second coordinate gives the required sets 
$V_P$.

For $x\in G$, let $P_{n}(x)$ be the member of $\mathcal{P}_{n}$ containing $%
x $. Completeness and (1)--(3) give a unique point 
\begin{equation*}
y(x)\in \bigcap_{n}\overline{V}_{P_{n}(x)}.
\end{equation*}%
The diameter condition and the fact that the $P_{n}(x)$ are clopen show that 
$y:G\rightarrow G$ is continuous, and its graph is contained in $%
\bigcap_{n}A_{n}$. Set $a(x)=xy(x)$ and $b(x)=y(x)^{-1}$. The definition of
the sets $A_{n}$ gives $a(x),b(x)\in C$, while $a(x)b(x)=x$. Both maps are
continuous.
\end{proof}

\begin{proposition}
\label{prop:continuous-lift} Let $G/N$ and $H/M$ be groups with a Polish
cover, where $G$ is non-Archimedean. Suppose also that $f:G/N\longrightarrow
H/M$ is a group homomorphism. Then the following assertions are equivalent:

\begin{enumerate}
\item $f$ is a Borel-definable homomorphism;

\item $f$ has a continuous lift $\varphi :G\rightarrow H$.
\end{enumerate}
\end{proposition}

\begin{proof}
A Borel map between Polish spaces is continuous after restriction to a dense 
$G_{\delta }$ set \cite[Theorem~8.38]{kechris_classical_1995}. Let $%
\psi:G\rightarrow H$ be a Borel lift of $f$, and choose such a set $%
C\subseteq G$ for $\psi$. Apply \Cref{lem:two-factor} and define 
\begin{equation*}
\varphi (x)=\psi (a(x))\psi (b(x)).
\end{equation*}%
This is continuous, since $a$ and $b$ take their values in $C$ and $\psi
|_{C}$ is continuous. Moreover, 
\begin{align*}
\varphi (x)M& =\psi (a(x))M\,\psi (b(x))M \\
& =f(a(x)N)f(b(x)N) \\
& =f(xN).
\end{align*}%
Thus $\varphi$ is a continuous lift of $f$.
\end{proof}

Note that the defect 
\begin{equation*}
d_{\varphi }(x,y)=\varphi (x)\varphi (y)\varphi (xy)^{-1}
\end{equation*}%
takes values in $M$, but \Cref{prop:continuous-lift} does not make the map $%
d_{\varphi }:G\times G\rightarrow M$ continuous with respect to the Polish
topology witnessing that $M$ is Polish.

\subsection{Exact completions}

Now let $\mathbf{PC}$ be the category whose objects are groups with a Polish
cover and whose morphisms are Borel-definable homomorphisms. As a direct
consequence of \Cref{prop:borel-graph} we obtain:

\begin{theorem}
\label{thm:main-completion} There is an equivalence of categories 
\begin{equation*}
\mathrm{Ex}(\mathbf{P})\simeq \mathbf{PC}
\end{equation*}%
mapping each Polish group $G$ to the group with a Polish cover $G/N$ where $%
N $ is the trivial subgroup of $G$.
\end{theorem}

\begin{proof}
As observed above, the objects of $\mathrm{Ex}\left( \mathbf{P}\right) $,
i.e., the equivalence relations in $\mathbf{P}$, can be identified as groups
with a Polish cover. Consider now groups with a Polish cover $G/N$ and $H/M$%
, corresponding to coset relations $E_{N}^{G}$ over $G$ and $E_{M}^{H}$ over 
$H$. A morphism $G/N\rightarrow H/M$ in $\mathrm{Ex}\left( \mathbf{P}\right) 
$ is a relation $T$ from $G$ to $H$ in $\mathbf{P}$ such that 
\begin{equation*}
TE_{N}^{G}=T=E_{M}^{H}T,\qquad E_{N}^{G}\leq T^{\circ }T,\qquad TT^{\circ
}\leq E_{M}^{H}.
\end{equation*}%
In particular, $T$ is a Polish subgroup of $G\times H$. The displayed
conditions define a unique map 
\begin{equation*}
f_{T}:G/N\longrightarrow H/M,\qquad f_{T}(gN)=hM\quad\text{whenever }%
(g,h)\in T.
\end{equation*}%
Indeed, totality gives the existence of $h$, while compatibility and
single-valuedness give independence of all choices. Since $T$ is a subgroup, 
$f_{T}$ is a homomorphism, and saturation gives $T=\widehat{\Gamma }(f_{T})$%
. Conversely, the lifted graph of any homomorphism $f:G/N\rightarrow H/M$ is
an entire functional relation as soon as it is a Polish subgroup of $G\times
H$.

By \Cref{prop:borel-graph}, these are precisely the Borel-definable
homomorphisms. Relational identities and composition correspond to identity
homomorphisms and composition, since the relational composite of two lifted
graphs is the lifted graph of the composite homomorphism. This yields the
claimed equivalence. Under the canonical embedding, $G$ is sent to $G/N$
where $N$ is the trivial subgroup of $G$.
\end{proof}

\begin{corollary}
The category $\mathbf{PC}$ of groups with a Polish cover is an exact
homological category.
\end{corollary}

\subsection{Regular subcategories}

Let $\mathcal{S}$ be a \emph{thick subcategory} of $\mathbf{P}$. Recall that 
$\mathcal{S}$ is thus the full subcategory spanned by a class of Polish
groups that is (essentially) closed under finite products, closed subgroups,
and quotients by closed normal subgroups, and closed under extensions in the
restricted sense of \Cref{Definition:thick}: whenever $1\rightarrow
N\rightarrow G\rightarrow H\rightarrow 1$ is an extension with $N$ and $H$
in $\mathcal{S}$ and $G$ is a subobject of an object of $\mathcal{S}$, then $%
G$ is in $\mathcal{S}$. By definition, an $\mathcal{S}$-group with a Polish
cover is a group with a Polish cover $G/N$ where both $G $ and the coset
relation of $N$ in $G$ are in $\mathcal{S}$. Since $\mathcal{S}$ is thick,
this is actually equivalent to the assertion that $G$ and $N$ are in $%
\mathcal{S}$. Indeed, suppose that $G$ and $N$ are in $\mathcal{S}$. By %
\Cref{Corollary:continuous-conjugation}, $E_{N}^{G}$ is a Polish subgroup of 
$G\times G$, and $G\times G$ is in $\mathcal{S}$; since the projection $%
E_{N}^{G}\rightarrow G$ is a regular epimorphism with kernel $N$, thickness
applied to the short exact sequence%
\begin{equation*}
1\longrightarrow N\longrightarrow E_{N}^{G}\longrightarrow G\longrightarrow 1
\end{equation*}%
yields $E_{N}^{G}\in \mathcal{S}$. Conversely, if $G$ and $E_{N}^{G}$ are in 
$\mathcal{S}$, the same short exact sequence yields $N\in \mathcal{S}$. The
following is therefore an immediate consequence of %
\Cref{Proposition:completion-thick} and \Cref{thm:main-completion}:

\begin{theorem}
\label{Theorem:main-completion-thick} Let $\mathcal{S}$ be a thick
subcategory of the category $\mathbf{P}$ of Polish groups. Then the
equivalence $\mathrm{Ex}\left( \mathbf{P}\right) \simeq \mathbf{PC}$
restricts to an equivalence between $\mathrm{Ex}\left( \mathcal{S}\right) $
and the full subcategory of $\mathbf{PC}$ spanned by $\mathcal{S}$-groups
with a Polish cover.
\end{theorem}

More generally, one can consider regular subcategories $\mathcal{S}$ of $%
\mathbf{P}$ that are not thick. In this case one obtains an analogous
description of $\mathrm{Ex}\left( \mathcal{S}\right) $ as a \emph{not
necessarily full} subcategory of $\mathbf{PC}$, in terms of $\mathcal{S}$%
-definable homomorphisms.

\begin{definition}
\label{Definition:S-definable} Let $\mathcal{S}$ be a regular subcategory of 
$\mathbf{P}$, and $G/N$ and $H/M$ be $\mathcal{S}$-groups with a Polish
cover. Then a group homomorphism $f:G/N\rightarrow H/M$ is $\mathcal{S}$%
-definable if it is Borel-definable and its lifted graph $\widehat{\Gamma }%
\left( f\right) $ is isomorphic to an object of $\mathcal{S}$ with respect
to its canonical Polish group topology.
\end{definition}

Now let $\mathcal{S}\mathbf{C}$ be the category that has $\mathcal{S}$%
-groups with a Polish cover as objects, and $\mathcal{S}$-definable
homomorphisms as morphisms. The proof of \Cref{thm:main-completion} yields
the following more general result, which generalizes %
\Cref{Theorem:main-completion-thick} as well as \cite[Theorem~6.14]%
{lupini_looking_2024}.

\begin{theorem}
\label{Theorem:main-completion-S} Let $\mathcal{S}$ be a regular subcategory
of $\mathbf{P}$. Then the equivalence $\mathrm{Ex}\left( \mathbf{P}\right)
\simeq \mathbf{PC}$ restricts to an equivalence $\mathrm{Ex}\left( \mathcal{S%
}\right) \simeq \mathcal{S}\mathbf{C}$.
\end{theorem}

\section{Better lifts \label{Section:better-lifts}}

In this section we show that, in certain situations, a Borel-definable
homomorphism between groups with a Polish cover has a lift with better
regularity than an arbitrary Borel function. For arbitrary (not necessarily
commutative) Polish groups, we generalize Proposition~5.3 and
Corollaries~6.19 and~6.20 of \cite{lupini_looking_2024}.

\subsection{A rectangular Baire category argument}

Suppose that $G/N$ and $H/M$ are groups with a Polish cover. Let $\varphi
:G/N\rightarrow H/M$ be a homomorphism, and let $f:G\rightarrow H$ be a
Borel lift of $\varphi $. Its \emph{multiplicative defect} is 
\begin{equation*}
\delta f:G\times G\longrightarrow M,\qquad \delta f(x,y)=f(x)f(y)f(xy)^{-1}%
\text{;}
\end{equation*}%
see \cite[Definition~5.1]{lupini_looking_2024}. We say that $f$ is \emph{%
approximately multiplicative} if $f(1)=1$ and 
\begin{equation*}
\delta f:G\times G\rightarrow M
\end{equation*}%
is continuous at $(1,1)$. This definition subsumes \cite[Definition~5.1]%
{lupini_looking_2024} in the particular case of abelian groups with a Polish
cover.

The following can be seen as the analogue for Polish groups that are not
necessarily abelian of \cite[Lemma~5.2]{lupini_looking_2024}.

\begin{lemma}
\label{Lemma:rectangular-fusion} Let $X$ and $Y$ be Polish spaces, let $H$
be a Polish group, let $P\subseteq X$ and $Q\subseteq Y$ be nonempty open
sets, and let 
\begin{equation*}
c:P\times Q\longrightarrow H
\end{equation*}%
be continuous with range contained in a Polish subgroup $M$ of $H$. Suppose
that $(K_{n})_{n\in \mathbb{N}}$ is a sequence of identity neighborhoods in $%
M$ such that 
\begin{equation*}
\overline{K}_{n}^{H}\,\cap M=K_{n}\qquad (n\in \mathbb{N}).
\end{equation*}%
Then there exist $a\in P$, $b\in Q$, open neighborhoods $P_{n}\subseteq P$
of $a$ and $Q_{n}\subseteq Q$ of $b$, and elements $m_{n}\in M$ such that 
\begin{equation*}
c(P_{n}\times Q_{n})\subseteq K_{n}m_{n}\qquad (n\in \mathbb{N}).
\end{equation*}
\end{lemma}

\begin{proof}
Fix compatible complete metrics on $X$ and $Y$. Put $P_{-1}=P$ and $Q_{-1}=Q$%
. Suppose that $P_{n-1}$ and $Q_{n-1}$ have been constructed. Since $M$ is
separable and $K_{n}$ is an identity neighborhood in $M$, there is a
countable set $D_{n}\subseteq M$ such that 
\begin{equation*}
M=\bigcup_{m\in D_{n}}K_{n}m.
\end{equation*}%
For $m\in D_{n}$, set 
\begin{equation*}
F_{m}=\{(x,y)\in P_{n-1}\times Q_{n-1}:c(x,y)\in \overline{K}_{n}^{\,H}m\}.
\end{equation*}%
These are relatively closed subsets of $P_{n-1}\times Q_{n-1}$, and they
cover that space. Hence by the Baire Category Theorem, some $F_{m_{n}}$ has
nonempty relative interior. Choose nonempty open sets $P_{n}\subseteq
P_{n-1} $ and $Q_{n}\subseteq Q_{n-1}$, of diameter at most $2^{-n}$, such
that 
\begin{equation*}
\overline{P}_{n}\subseteq P_{n-1},\qquad \overline{Q}_{n}\subseteq
Q_{n-1},\qquad \overline{P}_{n}\times \overline{Q}_{n}\subseteq F_{m_{n}}.
\end{equation*}

The nested closed sets $\overline{P}_{n}$ and $\overline{Q}_{n}$ have
diameters tending to zero. Completeness yields unique points 
\begin{equation*}
\{a\}=\bigcap_{n}\overline{P}_{n},\qquad \{b\}=\bigcap_{n}\overline{Q}_{n}.
\end{equation*}%
Since $\overline{P}_{n+1}\subseteq P_{n}$ and $\overline{Q}_{n+1}\subseteq
Q_{n}$, the sets $P_{n}$ and $Q_{n}$ are neighborhoods of $a$ and $b$.
Finally, because $c$ takes values in $M$, 
\begin{equation*}
c(P_{n}\times Q_{n})\subseteq (\overline{K}_{n}^{\,H}m_{n})\cap M=K_{n}m_{n}.
\end{equation*}
\end{proof}

\subsection{The lifting theorem}

As in \cite[Section~5]{lupini_looking_2024}, we deduce from %
\Cref{Lemma:rectangular-fusion} the following analogue of \cite[%
Proposition~5.3]{lupini_looking_2024}. Let $X,Y$ be Polish groups. Let $%
u:X\rightarrow Y$ be a function. Then $u$ is:

\begin{itemize}
\item \emph{locally continuous} if it is continuous on an identity
neighborhood of $X$;

\item (when $X$ is locally compact) \emph{locally bounded} if $u(C)$ has
compact closure in $Y$ for every compact subset $C$ of $X$ \cite%
{kehlet_cross_1984}.
\end{itemize}

\begin{theorem}
\label{Theorem:approximately-multiplicative-lift} Let $G/N$ and $H/M$ be
groups with a Polish cover, and let 
\begin{equation*}
\varphi :G/N\longrightarrow H/M
\end{equation*}%
be a homomorphism. Assume that $M$ has a basis of identity neighborhoods
that are closed in the subspace topology inherited from $H$. Suppose that $%
\varphi $ has a lift $f:G\rightarrow H$ that is Borel and continuous on an
identity neighborhood. Then:

\begin{enumerate}
\item $\varphi $ has an approximately multiplicative, locally continuous
Borel lift;

\item if $G$ is locally compact and $f$ is also locally bounded, then $%
\varphi $ has an approximately multiplicative, locally continuous, and
locally bounded Borel lift.
\end{enumerate}
\end{theorem}

\begin{proof}
Let $f:G\rightarrow H$ be a Borel lift of $\varphi $ that is continuous on
an open identity neighborhood $U$ of $G$. Choose an open identity
neighborhood $V$ such that $V^{2}\subseteq U$ and define 
\begin{equation*}
c:V\times V\longrightarrow M,\qquad c(u,v)=f(u)f(v)f(uv)^{-1}.
\end{equation*}%
As a map into $H$, the function $c$ is continuous.

Choose a symmetric identity-neighborhood basis $(K_{n})_{n\in \mathbb{N}}$
of $M$ such that 
\begin{equation*}
\overline{K}_{n}^{\,H}\cap M=K_{n}\qquad (n\in \mathbb{N}).
\end{equation*}%
Apply \Cref{Lemma:rectangular-fusion} to obtain $a,b\in V$, open
neighborhoods $P_{n},Q_{n}\subseteq V$ of $a,b$, and $m_{n}\in M$ such that 
\begin{equation}
c(P_{n}\times Q_{n})\subseteq K_{n}m_{n}.  \label{Equation:one-coset}
\end{equation}%
Define 
\begin{equation}
g(t)=f(b)f(ab)^{-1}f(atb)f(b)^{-1}\qquad (t\in G).  \label{Equation:new-lift}
\end{equation}%
This is a Borel map and $g(1)=1$. Moreover, 
\begin{align*}
g(t)M& =\varphi (bN)\varphi (abN)^{-1}\varphi (atbN)\varphi (bN)^{-1} \\
& =\varphi (tN),
\end{align*}%
so $g$ is a lift of $\varphi $. Since $ab\in U$, there is an identity
neighborhood $W$ in $G$ such that $aWb\subseteq U$. Equation~%
\eqref{Equation:new-lift} shows that $g|_{W}$ is continuous.

It remains to control the multiplicative defect of $g$. For $x,y$ close to
the identity, put 
\begin{equation*}
\begin{array}{ll}
c_{00}=c(a,b), & c_{x0}=c(ax,b), \\ 
c_{0y}=c(a,yb), & c_{xy}=c(ax,yb),%
\end{array}%
\end{equation*}%
and define 
\begin{equation*}
\begin{aligned} A_x&=c_{00}c_{x0}^{-1},&\qquad B_y&=c_{00}c_{0y}^{-1},\\
E_{x,y}&=c_{00}c_{xy}^{-1},& R_x&=f(ax)f(a)^{-1}. \end{aligned}
\end{equation*}%
The four identities 
\begin{equation*}
\begin{array}{ll}
f(ab)=c_{00}^{-1}f(a)f(b), & f(axb)=c_{x0}^{-1}f(ax)f(b), \\ 
f(ayb)=c_{0y}^{-1}f(a)f(yb), & f(axyb)=c_{xy}^{-1}f(ax)f(yb)%
\end{array}%
\end{equation*}%
give, after substitution in \eqref{Equation:new-lift}, the four-corner
identity 
\begin{equation}
\delta g(x,y)=f(a)^{-1}A_{x}R_{x}B_{y}R_{x}^{-1}E_{x,y}^{-1}f(a).
\label{Equation:four-corner}
\end{equation}

Let $T$ be an identity neighborhood in $M$. By %
\Cref{Corollary:continuous-conjugation}, the map 
\begin{equation*}
M\times H\times M\times M\longrightarrow M,\qquad (A,R,B,E)\longmapsto
f(a)^{-1}ARBR^{-1}E^{-1}f(a)
\end{equation*}%
is continuous at $(1,1,1,1)$. Hence there are identity neighborhoods $S$ in $%
M$ and $O$ in $H$ such that 
\begin{equation}
f(a)^{-1}SOSO^{-1}S^{-1}f(a)\subseteq T.
\label{Equation:neighborhood-control}
\end{equation}%
Choose $n$ so that $K_{n}K_{n}^{-1}\subseteq S$. For $x,y$ sufficiently
close to $1$, all four pairs 
\begin{equation*}
(a,b),\quad (ax,b),\quad (a,yb),\quad (ax,yb)
\end{equation*}%
belong to $P_{n}\times Q_{n}$. By \eqref{Equation:one-coset}, the
corresponding four values of $c$ belong to the same right translate $%
K_{n}m_{n}$. Therefore 
\begin{equation*}
A_{x},B_{y},E_{x,y}\in K_{n}K_{n}^{-1}\subseteq S.
\end{equation*}%
Furthermore, $R_{x}\rightarrow 1$ as $x\rightarrow 1$, since $f$ is
continuous at $a$. After shrinking the neighborhood of $1$ in $G$, we may
thus assume that $R_{x}\in O$. Equations \eqref{Equation:four-corner} and %
\eqref{Equation:neighborhood-control} imply that $\delta g(x,y)\in T$. Since 
$T$ was arbitrary, $\delta g$ is continuous at $(1,1)$ as an $M$-valued map.

(2) The definition of $g$ from $f$ above shows that when $f$ is locally
bounded, then so is $g$.
\end{proof}

\begin{corollary}
\label{Corollary:locally-compact} Suppose, in the setting of %
\Cref{Theorem:approximately-multiplicative-lift}, that $M$ is locally
compact. Then every homomorphism $G/N\rightarrow H/M$ that has a locally
continuous Borel lift has an approximately multiplicative, locally
continuous Borel lift.
\end{corollary}

\begin{proof}
The locally compact group $M$ has a basis of compact identity neighborhoods.
Their images in the Polish group $H$ are compact, hence closed, because the
inclusion $M\rightarrow H$ is continuous. Thus the hypothesis of %
\Cref{Theorem:approximately-multiplicative-lift} is automatic.
\end{proof}

\subsection{Locally compact covers}

We will apply some classical results about \emph{cross-sections} in the
context of topological groups. Cross sections for closed subgroups of
finite-dimensional locally compact groups were studied in \cite%
{mostert_local_1953,nagami_cross_1963,nagami_dimension-theoretical_1962,karube_local_1958,serre_extensions_1952,mackey_induced_1952}%
. Notice that on locally compact Polish spaces the Borel $\sigma $-algebra
and the Baire $\sigma $-algebra as defined in \cite{kehlet_cross_1984}
coincide.

\begin{lemma}
\label{Lemma:locally-compact-sections} Let $p:E\rightarrow G$ be a
surjective continuous homomorphism of locally compact Polish groups.

\begin{enumerate}
\item The map $p$ has a locally bounded Borel right inverse.

\item If $E$ has finite covering dimension, then $p$ has a locally bounded
Borel right inverse that is continuous on an open identity neighborhood in $%
G $.

\item If $\ker (p)$ is a Lie group, then $p$ has a locally bounded Borel
right inverse that is continuous on an open identity neighborhood in $G$.

\item If $E$ and $G$ are real Lie groups, then $p$ has a locally bounded
Borel right inverse that is real analytic on an open identity neighborhood
in $G$.
\end{enumerate}

In each case the right inverse may be chosen to map $1$ to $1$.
\end{lemma}

\begin{proof}
Assertion (1) is proved in \cite{kehlet_cross_1984}. Fix a locally bounded
Borel right inverse $b:G\rightarrow E$ given by that result. After replacing 
$b(x)$ with $b(x)b(1)^{-1}$, we may assume that $b(1)=1$.

For (2), a continuous local right inverse is provided by \cite%
{nagami_cross_1963}. For (3), it is provided by \cite[Th\'{e}or\`{e}me~1]%
{serre_extensions_1952}; see also \cite{gleason_spaces_1950}. For (4), a
real-analytic local right inverse is provided by \cite[Theorem~2.9.5]%
{varadarajan_lie_1984}. In each case, let $s:U\rightarrow E$ be the
resulting right inverse on an open identity neighborhood $U$ in $G$. After
replacing $s(x)$ with $s(x)s(1)^{-1}$, we may assume that $s(1)=1$.

Choose an open identity neighborhood $V$ such that $\overline{V}$ is compact
and $\overline{V}\subseteq U$, and define 
\begin{equation*}
r(x)= 
\begin{cases}
s(x), & x\in V, \\ 
b(x), & x\notin V.%
\end{cases}%
\end{equation*}
Then $r$ is a Borel right inverse for $p$, it maps $1$ to $1$, and it has
the required regularity on $V$. It is locally bounded: if $C\subseteq G$ is
compact, then 
\begin{equation*}
r(C)\subseteq s(\overline{V})\cup b(C),
\end{equation*}
whose closure is compact.
\end{proof}

The following corollary subsumes \cite[Corollary~6.19]{lupini_looking_2024}
in the case of abelian groups with a Polish cover. Its proof is the same as
the proof of \cite[Corollary~6.19]{lupini_looking_2024}, by using Lemma~\ref%
{Lemma:locally-compact-sections}, and replacing \cite[Theorem 6.18]%
{lupini_looking_2024} with Theorem~\ref{Theorem:main-completion-S} and
Proposition~\ref{prop:borel-graph}.

\begin{corollary}
\label{Corollary:locally-compact-Borel} Let $G/N$ and $H/M$ be groups with a
Polish cover and let 
\begin{equation*}
\varphi:G/N\longrightarrow H/M
\end{equation*}
be a group homomorphism. Suppose that $G$ and $M$, with their given Polish
group topologies, are locally compact. Then the following assertions are
equivalent:

\begin{enumerate}
\item $\varphi$ is Borel-definable;

\item $\varphi$ has a locally bounded Borel lift $G\rightarrow H$.
\end{enumerate}

If, furthermore, either both $G$ and $M$ have finite covering dimension or $%
M $ is a real Lie group, these conditions are equivalent to:

\begin{enumerate}
\item \setcounter{enumi}{2}

\item $\varphi $ has an approximately multiplicative, locally bounded, and
locally continuous Borel lift.
\end{enumerate}
\end{corollary}

\subsection{Lie covers}

By definition, a group with a Polish cover $G/N$ is a \emph{group with a }%
(real)\emph{\ Lie cover} if both $G$ and $N$ are real Lie groups. The next
result is the noncommutative version of \cite[Corollary~6.20]%
{lupini_looking_2024}. It is proved in the same fashion, by appealing to
Lemma~\ref{Lemma:locally-compact-sections}(4).

\begin{corollary}
\label{Corollary:Lie} Let $G/N$ and $H/M$ be groups with a real Lie cover,
and let 
\begin{equation*}
\varphi:G/N\longrightarrow H/M
\end{equation*}
be a group homomorphism. Then the following assertions are equivalent:

\begin{enumerate}
\item $\varphi$ is Borel-definable;

\item $\varphi $ has an approximately multiplicative Borel lift $%
G\rightarrow H$ that is real analytic on an open identity neighborhood in $G$%
.
\end{enumerate}
\end{corollary}

\bibliographystyle{amsplain}
\bibliography{bibliography}

\end{document}